\documentclass{article}

\usepackage[english]{babel}
\usepackage[utf8x]{inputenc}
\usepackage[T1]{fontenc}
\usepackage{enumitem}

\usepackage{amsmath}
\usepackage{amssymb}
\usepackage{amsthm}
\usepackage{graphicx}
\usepackage[colorinlistoftodos]{todonotes}
\usepackage[colorlinks=true, allcolors=black]{hyperref}

\newcommand{\PSL}{{\textrm{PSL}}}
\newcommand{\Inj}{{\textrm{Inj}}}
\newcommand{\IInj}{{\underline{\textrm{Inj}}}}
\newcommand{\N}{{\mathbb{N}}}
\newcommand{\Q}{{\mathbb{Q}}}
\newcommand{\R}{{\mathbb{R}}}
\newcommand{\h}{{\mathbb{H}^2}}
\newcommand{\hh}{{\mathbb{H}}}

\newcommand{\im}{{\textrm{Im}}}

\newcommand{\Id}{{\textrm{Id}}}
\newcommand{\Asymp}{\underset{n\to+\infty}{\asymp}}
\makeatletter
\newcommand*{\rom}[1]{\expandafter\@slowromancap\romannumeral #1@}
\makeatother

\setenumerate{label=(\roman*)}

{\theoremstyle {plain} \newtheorem{theorem}{Theorem}[section]}
{\theoremstyle {definition} \newtheorem{defi}[theorem]{Définition}}
{\theoremstyle {plain} \newtheorem{lemma}[theorem]{Lemma}}
\newtheorem{coro}[theorem]{Corollary}
\newtheorem{prop}[theorem]{Proposition}

\title{On links between horocyclic and geodesic orbits on geometrically infinite surfaces}
\author{Alexandre Bellis}
\date{July 09, 2017}
\begin{document}
\maketitle

\begin{abstract}
We study the topological dynamics of the horocycle flow $h_\R$ on a geometrically infinite hyperbolic surface $S$. Let $u$ be a non-periodic vector for $h_\R$ in $T^1S$. Suppose that the half-geodesic $u(\R^+)$ is almost minimizing and that the injectivity radius along $u(\R^+)$ has a finite inferior limit $\IInj(u(\R^+))$. 
We prove that the closure of $h_\R u$ meets the geodesic orbit along un unbounded sequence of points $g_{t_n}u$. Moreover, if $\IInj(u(\R^+)) = 0$, the whole half-orbit $g_{\R^+}u$ is contained in $\overline{h_\R u}$. 

When $\IInj(u(\R^+)) > 0$, it is known that in general $g_{\R^+}u \not\subset \overline{h_\R u}$. Yet, we give a construction where $\IInj(u(\R^+)) > 0$ and $g_{\R^+}u \subset \overline{h_{\R}u}$, which also constitutes a counter-example to Proposition $3$ of \cite{Ledrappier1997}.
\end{abstract}

\section{Introduction}

Among the curves of constant curvature in $\h$ are the geodesics of curvature zero and the horocycles of curvature one. They give rise to two flows with deep relations in the unitary tangent bundle $T^1\h$ : the geodesic flow and the horocycle flow respectively. Consider now a fuchsian group $\Gamma$ and the quotient surface $S:=\Gamma\backslash\h$. Both of these flows descend to the quotient $T^1S := \Gamma\backslash(T^1\h)$. We denote them by $g_\R$ and $h_\R$ respectively.

While orbits of the geodesic flow are known to be diverse, on the opposite, these of the horocycle flow tend to be rigid. It is illustrated by a result of G.~Hedlund in \cite{Hedlund1936} stating that when the  surface $S$ is compact, for every $u$ in $T^1S$, the orbit $h_\R u$ is dense in $T^1S$. In \cite{Eberlein1977}, P. Eberlein proves that this result comes from a fundamental link between horocyclic and geodesic orbits~: $h_\R u$ is not dense in the non-wandering set $\Omega_h$ of the horocycle flow if and only if the projection on $S$ of $g_{\R^+}u$, denoted by $u(\R^+)$, is almost minimizing (specifically : $\exists C > 0, \forall t > 0, d(u(0),u(t)) > t - C$).

As a consequence, he obtains that if the surface is geometrically finite (i.e. with a finitely generated fundamental group), if $u$ is in $\Omega_h$, then $h_\R u$ is either dense in $\Omega_h$ or periodic. This rigidity property was generalised by M. Ratner to Lie groups and unipotent actions in \cite{Ratner1991}. However, it does not extend to geometrically infinite surfaces (i.e. not geometrically finite). Indeed, $S$ is geometrically finite if and only if every horocyclic orbit in $\Omega_h$ is dense in $\Omega_h$ or periodic (see \cite{DalBo2010}).

In this paper, we are interested in the topological dynamics of the horocycle flow on geometrically infinite surfaces, where little is known. Untwisted hyperbolic flutes are the simplest examples of such surfaces and still contain many interesting behaviours (see \cite{Haas1996} or \cite{Coudene2010}). More precisely, we investigate the links between the geodesic flow and the horocycle flow. We associate to $x$ in $S$ the real number $\Inj(x)$ defined as the maximal radius of a ball centered at $x$ without self-intersection. 

The main result of this paper is :

\begin{theorem}\label{mainthm}
Let $\Gamma$ be a fuchsian group without elliptic elements and such that the quotient surface $S:=\Gamma\backslash\h$ is geometrically infinite. Consider $u$ in the non-wandering set $\Omega_h$ of $h_\R$ in $T^1S$. Suppose that $u(\R^+)$ is almost minimizing and that $h_\R u$ is not periodic and define $\IInj(u(\R^+)) := \liminf\limits_{t\to+\infty}\Inj(u(t))$.

If $\IInj(u(\R^+)) < +\infty$, then there exists a sequence of times $t_n$ converging to $+\infty$ such that $g_{t_n}u \in \overline{h_\R u}$ for all $n$. 

Moreover, if $\IInj(u(\R^+)) = 0$, then $g_{\R^+}u \subset \overline{h_\R u}$.
\end{theorem}

In particular, when $\IInj(u(\R^+)) < +\infty$ for every $u$ in $T^1S$ (O. Sarig calls such a surface \textit{weakly tame}, see \cite{Sarig2010}), then, if $h_\R u$ is not periodic, there always exists a positive time $t$ such that $g_t u \in \overline{h_\R u}$. As a corollary, we get :

\begin{coro}
Let $\Gamma$ be a fuchsian group with no elliptic nor parabolic element such that the quotient surface $S := \Gamma\backslash\h$ is geometrically infinite. If $S$ is weakly tame, then the horocycle flow does not admit any minimal set on $T^1S$.
\end{coro}

This corollary gives an easy way to construct surfaces without a minimal set for the horocycle flow. The first example of such a surface was given by M. Kulikov in \cite{Kulikov2004}. Later, theorems of non-existence were obtained in \cite{Matsu2016} and \cite{LoMass2017}.

In the setting of Theorem~\ref{mainthm}, we can ask : 

\textit{Question : is it possible to have $0 < \IInj(u(\R^+)) < +\infty$ and $g_{\R^+}u \subset \overline{h_\R u}$, as in the case $\IInj(u(\R^+)) = 0$ ?}

Clearly, if $h_\R u$ is not recurrent (i.e. it does not accumulate on itself), then $g_{\R^+}u \not\subset \overline{h_\R u}$. This implies in particular that $h_\R u$ is locally closed (there exists a neighborhood $V$ of $u$ such that $V \cap (\overline{h_\R u} - h_\R u) = \emptyset$) even though it is not closed. In \cite{Starkov1995}, A.N. Starkov gives an example of a surface $S$ satisfying the hypotheses of Theorem~\ref{mainthm} such that $0 < \IInj(u(\R^+)) < +\infty$ and $h_\R u$ is not recurrent. In \cite{Ledrappier1997}, F. Ledrappier generalizes this example to manifolds with \textit{bounded geometry} (see Proposition $3$ of \cite{Ledrappier1997}) : \textit{let $M$ be a manifold with bounded geometry and $u$ in $T^1M$ such that $u(\R^+)$ is asymptotically almost minimizing. Then, for every $t\in\mathbb{R}$, the strong stable leaf 
\[ W^{ss}(g_tu) := \{ v\in T^1M \,|\, d(g_{t+s}u,g_sv)\underset{s\to+\infty}{\longrightarrow} 0 \}\]
is locally closed. 
}

When $M$ is a hyperbolic surface $S$, this statement becomes : \textit{suppose $S$ has an injectivity radius everywhere above some positive constant $C$. Let $u$ be in $\Omega_h$ such that $u(\R^+)$ is almost minimizing.
Then $h_\R u$ is locally closed.
}

Actually, this proposition is false and I construct in section \ref{secbnded} the following counter-example which also answers the previous question :

\begin{theorem}\label{mainexple}
There exists a geometrically infinite surface $S$ with an injectivity radius everywhere above some positive constant $C$ (i.e. $S$ has bounded geometry), with $u$ in $\Omega_h$ satisfying :
\begin{enumerate}
	\item $u(\R^+)$ is almost minimizing.
	\item $g_{\R^+}u \subset \overline{h_\R u}$.
	\item $\IInj(u(\R^+)) < +\infty$.
\end{enumerate}  
In particular, $h_\R u$ is not locally closed. 
\end{theorem}

\textbf{Acknowledgements.} I thank my Ph.D advisor Françoise Dal'Bo for her numerous and precious advices about the redaction of this paper. I also thank Yves Coudène for the fruitful discussions.

\section{Notations and tools}\label{sectools1}

For two points $z$ and $z'$ in $\h$ and two points $\xi$ and $\eta$ in $\partial\h := \R \cup \{\infty\}$, we denote by $[z,z']$ the hyperbolic segment between $z$ and $z'$, by $[z,\xi)$ the half-geodesic going from $z$ to $\xi$ and by $(\eta,\xi)$ the geodesic between $\eta$ and $\xi$. We denote by $g_\R$ and $h_\R$ the geodesic, respectively horocycle flow in the unitary tangent bundle $T^1\h$. For any $u$ in $T^1\h$, the function symbol $u(t)$ refers to the projection on $\h$ of $g_t u$ and $u^+$ refers to the extremity in $\partial\h$ of the half-geodesic $u(\R^+)$.

Consider now two elements $u$ and $v$ in $T^1\h$. Suppose that $u(0) = z$ and $v(0) = z'$ and that $u^+ = v^+ = \xi$. Then there exists $t$ in $\R$ such that $g_t h_\R u = h_\R v$. The Busemann cocycle $B_\xi(z,z')$ centered at $\xi$ between $z$ and $z'$ is by definition the number $B_\xi(z,z') = t$.
Thus, the set $\{ z' \,|\, B_\xi(z,z') = 0 \}$ is the horocycle centered at $\xi$ passing through $z$.

\textbf{Level sets of isometries :}
The group $\PSL_2(\R)$ acts by orientation preserving isometries on $(\h,d)$. Let $\gamma \in \PSL_2(\R)$ be a hyperbolic isometry. We denote by $\gamma^-$ and $\gamma^+$ the points in $\partial\h$ that are its repulsive and attractive extremities respectively. Observe that a point $z\in\h$ is moved by $\gamma$ along the hypercycle passing through $\gamma^-$, z and $\gamma^+$, that we denote by $C_\gamma(z)$. For a positive integer $d$, the point $\gamma^d z$ belongs to the portion of $C_\gamma(z)$ between $z$ and $\gamma^+$.

Let $l_\gamma := \inf\limits_{z\in\h}d(z,\gamma z)$ be the translation length of $\gamma$, realised on its axis $(\gamma^-,\gamma^+)$. We have : 

\begin{prop}[Section $5$ of \cite{PaulPark2015}]\label{prophyper}
For any hyperbolic isometry $\gamma$ and any $z$ in $\h$, 
\[ \sinh \frac{d(z,\gamma z)}{2} = \cosh s \, \sinh \frac{l_\gamma}{2}
\]
where $s = d(z, (\gamma^-,\gamma^+))$.  
\end{prop}

When $\gamma$ is parabolic, we denote by $C_\gamma(z)$ the horocycle centered at the unique fixed point $\gamma^+ = \gamma^-$ of $\gamma$ and passing through $z$.

We have :

\begin{prop}[Section $6$ of \cite{PaulPark2015}]\label{proppara}
Consider a parabolic isometry $\gamma$ and pick any $z_0$ in $\h$. Denote by $l_\gamma(z_0)$ the distance $l_\gamma(z_0) = d(z_0,\gamma z_0)$. 

For any $z$ in $\h$, we have : 

\[ \sinh \frac{d(z, \gamma z)}{2} = e^s \sinh \frac{l_\gamma(z_0)}{2}
\] 
where $s = B_{\gamma^+}(z, C_\gamma(z_0))$.
\end{prop}

To prove our theorems, we will translate the dynamics of the horocycle flow on $\Gamma\backslash(T^1\h)$ in terms of the action of $\Gamma$ on $\h$ and $\partial\h$ using the following proposition.

\begin{prop}[Proposition $2.1$, Section \rom{5} of \cite{DalBo2010}]\label{equivcle}
	Take a vector $u$ in $T^1S$ and a positive real number $r$. Note $\tilde{u}$ a lift of $u$ in $T^1\h$ and suppose that $\tilde{u}^+$ is not fixed by any element of $\Gamma$ except the identity. Then 
	\[ \left( g_ru \in \overline{h_\R u}-h_\R u \right) \]
	\[ \Updownarrow \]
	\[ \left( \exists (\alpha_n)_\N \in \Gamma^\mathbb{N} \textrm{ s.t. } \alpha_n\tilde{u}^+ \underset{n\to+\infty}{\longrightarrow} \tilde{u}^+ \textrm{ and } B_{\tilde{u}^+}(\alpha_n^{-1}i,\tilde{u}(0)) \underset{n\to+\infty}{\longrightarrow} B_{\tilde{u}^+}(i,\tilde{u}(r))\right). \]
\end{prop}

\section{Proof of theorem~\ref{mainthm}}\label{secproof}

Let us first give a precise definition of the injectivity radius.

\begin{defi}[Injectivity radius] 
Let $S := \Gamma\backslash\h$ be a hyperbolic surface. The injectivity radius of $S$ at $p$ is : 
 \[ \Inj(p) := \min\limits_{\substack{\gamma\in\Gamma \\ \gamma\neq id}}d(\tilde{p},\gamma\tilde{p}) \]
 where $\tilde{p}$ is any lift of $p$ to $\h$.
\end{defi}

\begin{proof}[Proof of Theorem~\ref{mainthm}]
Consider a lift $\tilde{u}$ in $T^1\h$ of $u$. Up to conjugacy, we can suppose $\tilde{u}^+=\infty$ and $\tilde{u}(0) = i$.

Note that with our choice of lift $\tilde{u}$ of $u$, the equivalence of Proposition~\ref{equivcle} becomes : 

\begin{equation}\label{equi}
(g_r u \in \overline{h_\R u} - h_\R u) \Longleftrightarrow \left( \exists (\alpha_n)_\N \in \Gamma^\N,\, \alpha_n \infty \to \infty \textrm{ and } B_\infty(\alpha_n^{-1}i,i) \to r \right) 
\end{equation}

The key of this proof is to find elements $\alpha_n$ in $\Gamma$ on  which to apply this equivalence.

\textbf{Step 1 :}
\textit{There is a sequence of points $(q_n)_n$ going to $\infty$ on the half-geodesic $[i,\infty) = [\tilde{u}(0),\tilde{u}^+)$ and a sequence of elements $\gamma_n$ in $\Gamma$ that are all different such that}
\begin{enumerate}
\item \label{step11} $d(q_n,\gamma_nq_n) \underset{n\to+\infty}{\longrightarrow} \IInj(u(\R^+)).$
\item \label{step12} \textit{For every sequence of positive integers $(k_n)_\N$ we have $\gamma_n^{k_n}\infty \underset{n\to+\infty}{\longrightarrow} \infty.$}
\end{enumerate}

\begin{proof}[Proof of Step $1$]
The hypothesis $\IInj(u(\R^+)) < +\infty$ of Theorem~\ref{mainthm} gives us a sequence of points $q_n$ going to $\infty$ in $\partial\h$ along the half-geodesic $[i,\infty)$ and a sequence of elements $(\gamma_n)_{n\geq 0}$ in $\Gamma - \Id$ satisfying condition \ref{step11} of Step $1$. Since neither $g_\R u$ nor $h_\R u$ is periodic, these elements $\gamma_n$ are all different.

Consider a subsequence of $(\gamma_n)_\N$, that we still denote by $(\gamma_n)_\N$, such that $\lim\limits_{n\to+\infty} \gamma_n^- = \eta$ and $\lim\limits_{n\to+\infty} \gamma_n^+ = \xi$.

Suppose first that $\eta \neq \xi$. In this case, for any $z$ in $\h$, the distance $s_n = d(z,(\gamma_n^-,\gamma_n^+))$ is bounded from above. Thus, since $l_{\gamma_n} \leq d(q_n, \gamma_n q_n)$ for every $n$, Proposition~\ref{prophyper} implies that for any $z$ in $\h$, the elements $\gamma_n z$ stay in a compact. This contradicts the discreteness of $\Gamma$.

Suppose now that $\eta = \xi \neq \infty$. If the elements $\gamma_n$ are hyperbolic, consider the half-geodesic $[p,\xi)$ starting from a point $p$ on $(0,\infty)$ and orthogonal to $(0,\infty)$ (if $\eta = \xi = 0$, consider any point $p$ on $(0,\infty)$). For $n$ big enough, we have $d(p,(\gamma_n^-,\gamma_n^+)) < d(q_n,(\gamma_n^-,\gamma_n^+))$. Thus, Proposition~\ref{prophyper} implies that $d(p,\gamma_n p) < d(q_n,\gamma_n q_n)$. Since the latter is bounded from above, we get again a contradiction with the discreteness of $\Gamma$. Finally, if the elements $\gamma_n$ are parabolic, for any $z$ and any $z_0$ in $\h$, we eventually have $B_{\gamma_n^+}(z,C_{\gamma_n}(z_0)) < B_{\gamma_n^+}(q_n,C_{\gamma_n}(z_0))$. Thus, Proposition~\ref{proppara} implies that $d(z,\gamma_n z) < d(q_n, \gamma_n q_n)$, which gives again a contradiction with the discreteness of $\Gamma$. 

In conclusion,
\begin{equation}\label{egalitelim} \eta = \xi = \infty. 
\end{equation} 

Choose now the following orientation for the elements $\gamma_n$.
\begin{itemize}
\item If $\gamma_n$ is hyperbolic, choose $|\gamma_n^-| \leq |\gamma_n^+|.$ 
\item If $\gamma_n$ is parabolic, choose it such that $|\gamma_n \infty| > |\gamma_n^+|.$
\end{itemize}

This choice of orientation combined with (\ref{egalitelim}) concludes Step $1$. 
\end{proof}

\textbf{Step 2 :}
\textit{Fix $\epsilon > 0$ and an interval $I$ of $\R^+$ of length $\IInj(u(\R^+)) + \epsilon$. If $n$ is big enough, there exists an integer $k_n$ such that $B_\infty(\gamma_n^{-k_n}i,i)$ belongs to $I$.}

\begin{proof}[Proof of Step $2$]
For a positive integer $k$ we put : 

\[ r_{n,k} := B_\infty(\gamma_n^{-k}i,i).
\]

We have $r_{n,k} = \sum\limits_{l=0}^{k-1} B_\infty(\gamma_n^{-k+l}i,\gamma_n^{-k+l+1}i)$.

Let us prove that for $n$ big enough, each step $B_\infty(\gamma_n^{-k+l}i,\gamma_n^{-k+l+1}i)$ is smaller than $\Inj(u(\R^+)) + \epsilon$.

For now, we admit the following lemma : 

\begin{lemma}\label{existp}
For all positive integers $n$ and nonpositive integer $a$, there exists a point $p_{n,a}$ in $\h$ satisfying : 
\begin{enumerate}
 \item \label{lem1} $d(p_{n,a}, \gamma_n p_{n,a}) = d(q_n, \gamma_n q_n)$
 \item \label{lem2} $B_\infty(p_{n,a}, \gamma_n^{a} i) = -d(\gamma_n^{a} i, p_{n,a})$
\end{enumerate}
\end{lemma}

For convenience, set $s_{n,a} := d(\gamma_n^a i, p_{n,a})$ and $A_{k,l} := B_\infty(\gamma_n^{-k+l}i,\gamma_n^{-k+l+1}i)$. Using Lemma~\ref{existp}, we compute : 

\begin{equation*}
\begin{split}
A_{k,l} =~ &B_\infty(\gamma_n^{-k+l}i,\gamma_n^{-1}p_{n,-k+l+1}) +  B_\infty(\gamma_n^{-1}p_{n,-k+l+1},p_{n,-k+l+1}) \\ 
&+ B_\infty(p_{n,-k+l+1},\gamma_n^{-k+l+1}i) \\ 
\leq~ &d(\gamma_n^{-k+l}i,\gamma_n^{-1}p_{n,-k+l+1}) + d(\gamma_n^{-1}p_{n,-k+l+1},p_{n,-k+l+1}) - s_{n,-k+l+1} \\
=~ &s_{n,-k+l+1} + d(q_n,\gamma_n q_n) - s_{n,-k+l+1} \\
=~ &d(q_n,\gamma_n q_n)
\end{split}
\end{equation*}

As $d(q_n,\gamma_n q_n)$ is eventually smaller than $\IInj(u(\R^+)) + \epsilon$, we obtain that for $n$ big enough :  

\begin{equation}
A_{k,l} = B_\infty(\gamma_n^{-k+l}i,\gamma_n^{-k+l+1}i) \leq \IInj(u(\R^+)) + \epsilon. 
\end{equation}

Thus, $r_{n,k}$ is a sum of $k$ terms $A_{k,l}$, for $l=0 \dots k-1$, all smaller than $\IInj(u(\R^+)) + \epsilon$ which is the length of the interval $I$ of $\mathbb{R}^+$. We now use the fact that since $\gamma_n^{-k}i \underset{k\to+\infty}{\longrightarrow} \gamma_n^- \neq \infty$ for every $n$, we have $r_{n,k} \underset{k\to+\infty}{\longrightarrow} +\infty$ for every $n$. Thus, there exists an integer $k_n$ such that $r_{n,k_n}$ belongs to $I$.

This concludes Step $2$.
\end{proof}

\textit{End of the proof of Theorem~\ref{mainthm}.} Take the elements $\gamma_n$ given by Step~$1$. Fix an $\epsilon > 0$, chosen arbitrarily small, and an interval $I$ of $\R^+$ of length $\IInj(u(\R^+)) + \epsilon$. Consider the sequence of positive integers $(k_n)_\N$ given by Step~$2$. Since the numbers $B_\infty(\gamma_n^{-k_n}i,i)$ eventually all belong to $I$, the sequence of numbers $(B_\infty(\gamma_n^{-k_n}i,i))_\N$ admits an accumulation point $r$ in $I$. Thus, setting $\alpha_n := \gamma_n^{k_n}$ and applying the equivalence (\ref{equi}), we get that $g_r u \in \overline{h_\R u} - h_\R u$.

Now, applying the same argument to a partition $(I_l)_\N$ of $\R^+$ in intervals of lengths $\IInj(u(\R^+)) + \epsilon$, we get a sequence of times $(t_l)_\N$, where each $t_l$ is in $I_l$, and such that $g_{t_l}u$ belongs to $\overline{h_\R u}$ for every $l$. 

Moreover, if $\IInj(u(\R^+)) = 0$, as the intervals $I_l$ will be of length $0 + \epsilon$ for an $\epsilon > 0$ arbitrarily small, we get $g_{\R^+}u \subset \overline{h_\R u}$.

\end{proof}

\begin{proof}[Proof of Lemma~\ref{existp}.]
Take an isometry $\gamma_n$ as in Step $1$.

Observe that a point in $C_{\gamma_n}(q_n) \cap [ \gamma_n^ai, \infty)$ would satisfy condition \ref{lem1} and also condition \ref{lem2} according to Proposition~\ref{prophyper} and Proposition~\ref{proppara}. We prove that this intersection is not empty.

Suppose that $\gamma_n$ is hyperbolic. There are two cases. The first case is when the signs of $\gamma_n^-$ and $\gamma_n^+$ are opposed. Since $\lim\limits_{n\to+\infty} \gamma_n^- = \lim\limits_{n\to+\infty} \gamma_n^+ = \infty$, the graph of $C_{\gamma_n}(i)$ is below the one of $C_{\gamma_n}(q_n)$. So every geodesic starting from a point $z$ on $C_{\gamma_n}(i)$ and ending at $\infty$ has an intersection with $C_{\gamma_n}(q_n)$. This is true in particular when $z = \gamma_n^a i$.

The second case is when $\gamma_n^-$ and $\gamma_n^+$ have same sign. Observe then that we eventually have $d(i,(\gamma_n^-,\gamma_n^+)) > d(q_n,(\gamma_n^-,\gamma_n^+))$, because the converse would contradict the discretness of $\Gamma$ using Proposition~\ref{prophyper}. So, according to the same proposition, the graph of $C_{\gamma_n}(q_n)$ is below the one of $C_{\gamma_n}(i)$ if $n$ is big enough. Observe now that as $a$ is a nonpositive integer, the point $\gamma_n^a i$ belongs to the portion of $C_{\gamma_n}(i)$ between $i$ and $\gamma_n^-$, and that for any point $z$ in this portion of $C_{\gamma_n}(i)$, the intersection $C_{\gamma_n}(q_n) \cap [z, \infty)$ is not empty.

If $\gamma_n$ is parabolic, the proof is similar to the second case, when replacing $d(i,(\gamma_n^-,\gamma_n^+))$ by $B_{\gamma_n^+}(i,C_{\gamma_n}(z_0))$ and $d(q_n,(\gamma_n^-,\gamma_n^+))$ by $B_{\gamma_n^+}(q_n,C_{\gamma_n}(z_0))$ and using Proposition~\ref{proppara}. This concludes the proof. 
\end{proof}

\section{An example to prove theorem \ref{mainexple}}\label{secbnded}

Recall that the Dirichlet domain centered at $i$ of a fuchsian group $\Gamma$ with no elliptic element fixing $i$ is defined by :
\[ D_i(\Gamma) := \bigcap_{\substack{\gamma\in\Gamma \\ \gamma\neq id}} \hh_i(\gamma),
\]
where $\hh_i(\gamma) := \{ z \in \h ~|~ d(z,i) \leq d(z,\gamma(i)) \}$. 

The main interest for us of looking for this domain is the following classical fact (see section \rom{1}, Proposition $4.9$ of \cite{DalBo2010}) : \textit{if $u \in \Gamma\backslash(T^1\h)$ and if for some lift $\tilde{u}$ in $T^1\h$ the point $\tilde{u}^+$ belongs to $\overline{D_i(\Gamma)} \cap \partial\h$, then $u(\R^+)$ is almost minimizing.}

Let us now construct our example. Consider, for any rational number $q \in [4, +\infty)$ and any $n$ in $\N$, the hyperbolic isometry : 
\[ g_{q,n} := \begin{pmatrix} \sqrt{q} & (1-q)r_n \\[2pt]
                          -\frac{1}{r_n} & \sqrt{q} 
                        \end{pmatrix},
\]

where $(r_n)_\N$ is the sequence of real numbers defined by : 

\[ 
\left\{
\begin{array}{l}
  r_1 = 2 \\
  r_n = 3r_{n-1}, \forall n\geq 2. 
\end{array}
\right.
\]

Let $F_1 := \{ g_{q,n},\, q \in \Q \cap [4, +\infty),\, n \in \N \}$. 

We now conjugate the isometries $g_{4,n}$ by the isometries $T_{q,n} := \begin{pmatrix} 1 & t_{q,n} \\
                   0 & 1 \end{pmatrix} $,
for any rational number $q \in (1, 4)$ and any $n$ in $\N$, with $t_{q,n} := -r_n(\sqrt{q} - 2)$. We have : 
\begin{align*}
h_{q,n} &:= T_{q,n}^{-1} g_{4,n} T_{q,n} \\
        &= \begin{pmatrix} 4 - \sqrt{q} & r_n[(\sqrt{q}-2)^2 - 3] \\[2pt]
                              -\frac{1}{r_n} & \sqrt{q} \end{pmatrix}.
\end{align*}                              

Set $F_2 := \{ h_{q,n},\, q \in \Q \cap (1,4),\, n \in \N \}$. Define, for every $q \in \Q \cap (1,+\infty)$, the hyperbolic isometry $f_{q,n}$ by :
\begin{itemize}
\item $f_{q,n} := h_{q,n}$ if $q \in \Q \cap (1,4)$,
\item $f_{q,n} := g_{q,n}$ if $q \in \Q \cap [4,+\infty)$,
\end{itemize}
and set $F := F_1 \cup F_2 = \{ f_{q,n}, q \in \Q \cap (1,+\infty) \}$. 

For any non-elliptic isometry $\gamma$ in $\PSL_2(\R)$, define $\partial\hh_i(\gamma)$ to be the perpendicular bisector of the segment $[i,\gamma i]$. Also denote by $c(\gamma)$ the center of the euclidean half-circle $\partial\hh_i(\gamma)$ and by $e_l(\gamma)$, with $l=1,2$, the extremities in $\partial\h$ of $\partial\hh_i(\gamma)$. Denote also by $C_{(\eta,\xi)}(z)$ the hypercycle with extremities $\eta$ and $\xi$ in $\partial\h$ and passing through $z$ in $\h$.
 
The following key proposition gives us all the necessary information about the perpendicular bisectors $\partial\hh_i(\gamma)$ (see section~\ref{proofcv} for the proof).

\begin{prop}\label{propcv}
For every $q \in ( 1, +\infty )$ we have the following :
\begin{enumerate}
 \item \label{cv1} $\lim\limits_{n\to+\infty} c(f_{q,n}) = -\infty$ and $\lim\limits_{n\to+\infty} c(f_{q,n}^{-1}) = +\infty$.
 \item \label{cv2} $\lim\limits_{n\to+\infty} e_l(f_{q,n}) = -\infty$ and $\lim\limits_{n\to+\infty} e_l(f_{q,n}^{-1}) = +\infty$ for $l=1,2$.
 \item \label{cv3} Each perpendicular bisector $\partial\hh_i(f_{q,n}^{-1})$ is below the hypercycle $C_{(0,\infty)}(1+i)$.
\end{enumerate}
\end{prop}

Using Proposition~\ref{propcv}, our goal is to extract a sequence of elements $\gamma_m$ of $F$ such that all the perpendicular bisectors $\partial\hh_i(\gamma_m)$ are disjoint and which contains an infinite number of elements $f_{q,n}$ for every rational number $q \in (1,+\infty)$.

Consider any bijection $\psi : \N \mapsto \Q \cap (1,+\infty) \times \N$ and set $\psi(m) = (q_m,\psi_2(m))$. Observe that for every $q$ in $(1,+\infty)$, there is an infinite number of elements $q_m$ such that $q_m = q$. We can now write $F = \{ f_{q_m,\psi_2(m)}, m \in \N \}$. 

\textbf{Choice of the elements $\gamma_m$.} Put $\gamma_0 := f_{q_0,0}$ and set $2C$ to be the distance between the geodesic $(0,\infty)$ and the hypercycle $C_{(0,\infty)}(1+i)$. Then, for every $m \geq 1$, we define $\gamma_m$ by induction. We ask that $\gamma_m = f_{q_m,n_m}$ where $n_m$ is the smallest integer among the integers $p$ satisfying : 

\begin{enumerate}
\item $|e^\epsilon_k(f_{q_m,p})| > |e^\epsilon_l(\gamma_{m-1})|$, for $k=1,2$, $l=1,2$ and $\epsilon = \pm 1$.
\item $d(\partial\hh_i(f_{q_m,p}),\partial\hh_i(f_{q_m,p}^{-1})) \geq 2C$
\item $d(\partial\hh_i(f_{q_m,p}^{\pm 1}),\partial\hh_i(f_{q_s,n_s}^{\pm 1})) \geq C$ for all $s < m$.
\end{enumerate}

Such a sequence $(\gamma_m)_\N$ exists according to Proposition~\ref{propcv}. We now set $\Gamma := < \gamma_m,~m \in \N >$ and prove that $\Gamma$ answers Theorem~\ref{mainexple}.

By a classic ping-pong argument (see~\cite{DalBo2010}), the group $\Gamma$ is discrete and free. Moreover, its Dirichlet domain centered at $i$ is : 

\[ D_i(\Gamma) = \bigcap\limits_{m\in\mathbb{N}}\hh_i(\gamma_m) \cap \bigcap\limits_{m\in\mathbb{N}}\hh_i(\gamma_m^{-1}).  \]

Fix $\tilde{u}$ in $T^1\h$ such that $\tilde{u}(0) = i$ and $\tilde{u}^+ = \infty$ and consider its projection $u$ to $T^1S := \Gamma\backslash(T^1\h)$. Since $\infty$ is in $\overline{D_i(\Gamma)} \cap \partial\h$, the half-geodesic $u(\R^+)$ is almost minimizing. Hence (i) of the theorem.

To prove condition (ii), we use Proposition~\ref{equivcle}. Observe that it follows directly from the definition of $f_{q,n}$ that for every rational number $q$ in $(1,+\infty)$, we have 

\begin{equation}
f_{q,n}^{-1} \infty \underset{n\to+\infty}{\longrightarrow} \infty.
\end{equation}

Moreover, since $\im(f_{q,n}(i)) \underset{n\to+\infty}{\longrightarrow} \frac{1}{q}$, we have :

\begin{equation} B_\infty(f_{q,n}i,i) \underset{n\to+\infty}{\longrightarrow} B_\infty(\frac{i}{q},i) = |\ln(q)|.
\end{equation} 

Now, for every $q \in \Q \cap (1,+\infty)$ there is an infinite number of elements $f_{q,n}$ in $(\gamma_m)_\N$, thus in $\Gamma$. So, according to ($4$), ($5$) and Proposition~\ref{equivcle}, it follows that all the elements $g_{|\ln(q)|}u$ belong to $\overline{h_\R u}$. Hence, $g_{\R^+}u$ is included in $\overline{h_\R u}$ and we get (ii) of the theorem.

Finally, fix $z$ in the interior of $D_i(\Gamma)$ and any $\gamma = \gamma_{m_1}^{i_1}\dots\gamma_{m_k}^{i_k}$ in $\Gamma$ different from the identity, written as a reduced word in the letters $\gamma_m$. 

If $k=1$, then $\gamma = \gamma_{m_1}^{i_1}$ and 
\begin{align*}
  d(z,\gamma z) &= d(z,\gamma^{-1} z) \\
                &\geq \max(d(z,\partial\hh_i(\gamma_{m_1})), d(z,\partial\hh_i(\gamma_{m_1}^{-1}))) 
\end{align*}

Since $d(\partial\hh_i(\gamma_{m_1}), \partial\hh_i(\gamma_{m_1}^{-1})) \geq 2C$, we obtain that $d(z,\gamma z) \geq C$.

If $k > 1$, 

\begin{align*}
  d(z,\gamma z) &= d(z,\gamma_{m_1}^{i_1}\dots\gamma_{m_k}^{i_k} z) 
  = d(\gamma_{m_1}^{-i_1} z,\gamma_{m_2}^{i_2}\dots\gamma_{m_k}^{i_k} z) \\
  &\geq d(\partial\hh_i(\gamma_{m_1}^{\pm 1}),\partial\hh_i(\gamma_{m_2}^{\pm 1})) \\
  &\geq C.
\end{align*}

It follows that the injectivity radius on $S$ is everywhere greater than $C$. Finally, since all the axes of the elements $\gamma_m$ in $\Gamma$ intersect the half-geodesic $[i,\infty)$, and since their translation length $l_{\gamma_n}$ is constant by definition of the elements $f_{q,n}$, the injectivity radius $\IInj(u(\R^+))$ is also finite. So $C \leq \IInj(u(\R^+)) < +\infty$. Hence condition (iii) of Theorem~\ref{mainexple}. This completes the proof. \qed

\section{Proof of Proposition \ref{propcv}}\label{proofcv}

Using the classic formula :

\[ \forall a,b \in \h,\, \sinh\frac{d(a,b)}{2} = \frac{|a-b|}{2\sqrt{\im(a)\im(b)}},
\]
we get :

\begin{prop}\label{eqperp}
Consider a point $P= R + iI$ in $\h$. The equation of the perpendicular bisector of the hyperbolic segment between $i$ and $P$ is : 

\begin{equation*}
\left( x + \frac{R}{I-1} \right)^2 + y^2 = I\left( 1 + \frac{R^2}{(I-1)^2}\right)
\end{equation*} 
\end{prop}

For the following calculations, we distinguish the case $q \in [4,+\infty)$ from the case $q \in (1,4)$.

\textbf{Case 1: fix $q \in [4, +\infty)$.} We have $f_{q,n} i = R_{q,n} + iI_{q,n}$ where 
\[ R_{q,n} := \frac{\sqrt{q}(1-q)r_n - \frac{\sqrt{q}}{r_n}}{q + \frac{1}{r_n^2}}
\]
and 
\[ I_{q,n} := \frac{1}{q + \frac{1}{r_n^2}}.
\] 

Observe that we have $r_n \underset{n\to+\infty}{\longrightarrow} +\infty$. Thus, as $n$ converges to $+\infty$, the quantity $R_{q,n}$ is equivalent to $\frac{\sqrt{q}(1-q)r_n}{q} = r_n\frac{(1-q)}{\sqrt{q}}$, where the number $\frac{1-q}{\sqrt{q}}$ is different from $0$, and the quantity $I_{q,n}$ is equivalent to $\frac{1}{q}$. So, applying Proposition~\ref{eqperp}, we get the following asymptotic equivalence : 

\begin{equation*}
c(f_{q,n}) = -\frac{R_{q,n}}{I_{q,n} - 1} \Asymp -r_n\sqrt{q}. 
\end{equation*}

So the centers $c(f_{q,n})$ converge to $-\infty$.

Let us now study the radii of the geodesics $\partial\hh_i(f_{q,n})$. We have 

\begin{equation*}
\sqrt{I_{q,n}}\left( 1 + \frac{R_{q,n}^2}{(I_{q,n}-1)^2} \right)^{\frac{1}{2}} \Asymp \frac{1}{\sqrt{q}}\frac{R_{q,n}}{I_{q,n} - 1} = -\frac{1}{\sqrt{q}}c(f_{q,n})
\end{equation*}

where $\frac{1}{\sqrt{q}}$ belongs to $(0,\frac{1}{2} ]$. 

Thus, according to Proposition~\ref{eqperp}, the extremities $e_l(f_{q,n})$, for $l=1,2$, of the geodesics $\partial\hh_i(f_{q,n})$ converge to $-\infty$, and since $\frac{1}{\sqrt{q}} \leq \frac{1}{2}$, all these geodesics are below the hypercycle $C_{(0,\infty)}(-1+i)$. 

We now study the case of $\partial\hh_i(f_{q,n}^{-1})$. We have :

\[
 f_{q,n}^{-1}i = \frac{\sqrt{q}(q-1)r_n + \frac{\sqrt{q}}{r_n}}{q + \frac{1}{r_n^2}} + i\frac{1}{q + \frac{1}{r_n^2}}. 
\]

So we observe that the real part of $f_{q,n}^{-1}i$ is the negative of the real part of $f_{q,n}i$. So the geodesics $\partial\hh_i(f_{q,n}^{-1})$ and $\partial\hh_i(f_{q,n})$ are symetric with respect to the imaginary axis. In particular, they are below the hypercycle $C_{(0,\infty)}(1+i)$.

\textbf{Case 2: fix $q \in \Q \cap (1,4)$.} We have $f_{q,n}i = R_{q,n} + iI_{q,n}$ where
\[ R_{q,n} := \frac{r_n\sqrt{q}((\sqrt{q} - 2)^2 - 3) - \frac{1}{r_n}(4-\sqrt{q})}{q + \frac{1}{r_n^2}}
\]
and 
\[ I_{q,n} = \frac{1}{q + \frac{1}{r_n^2}}. \]

Observe that since $q \in (1,4)$, the number $\sqrt{q}((\sqrt{q} - 2)^2 - 3)$ is different from $0$. So as $n$ goes to $+\infty$, we have the equivalences :
 \begin{align*}
  R_{q,n} &\Asymp r_n\frac{(\sqrt{q}-2)^2-3}{\sqrt{q}}
\end{align*} 

and 

\[ I_{q,n} \Asymp \frac{1}{q}.
\]

Thus, according to Proposition~\ref{eqperp}

\[
c(f_{q,n}) = -\frac{R_{q,n}}{I_{q,n} - 1} \Asymp -r_n\sqrt{q} \frac{(\sqrt{q}-2)^2-3}{1-q} 
\]

where $\sqrt{q} \frac{(\sqrt{q}-2)^2-3}{1-q}$ is a real number greater than $2$. Thus, these centers converge to $-\infty$.

Let us now study the radii of the geodesics $\partial\hh_i(f_{q,n})$. We have

\[
\sqrt{I_{n,q}}\left( 1 + \frac{R_{n,q}^2}{(I_{n,q}-1)^2} \right)^{\frac{1}{2}} \Asymp \frac{1}{\sqrt{q}}\frac{R_{n,q}}{I_{n,q} - 1} = -\frac{1}{\sqrt{q}}c(f_{q,n})
\]

where $\frac{1}{\sqrt{q}}$ belongs to $(\frac{1}{2},1)$. 
Thus, again, the extremities  $e_l(f_{q,n})$, for $l=1,2$, of the geodesics $\partial\hh_i(f_{q,n})$ converge to $-\infty$ as $n$ goes to $+\infty$.

We now study the case of $f_{q,n}^{-1}$ for $q \in (1,4)$. We have $f_{q,n}^{-1}i = R_{q,n} + iI_{q,n}$ where 

\[ R_{q,n} = \frac{-r_n(4-\sqrt{q})[(\sqrt{q} - 2)^2 - 3] + \frac{\sqrt{q}}{r_n}}{(4-\sqrt{q})^2 + \frac{1}{r_n^2}}
\] 

and 

\[ I_{q,n} = \frac{1}{(4-\sqrt{q})^2 + \frac{1}{r_n^2}}.
\]

Since $q \in (1, 4)$, the number $(4-\sqrt{q})[(\sqrt{q} - 2)^2 - 3]$ is different from $0$. Thus, 

\[
R_{q,n} \Asymp r_n \frac{(\sqrt{q} - 2)^2 -3}{\sqrt{q} - 4}.
\]

Since
\[ I_{q,n} \Asymp \frac{1}{(4-\sqrt{q})^2},
\]
we have : 

\[
c(f_{q,n}^{-1}) = -\frac{R_{q,n}}{I_{q,n} - 1} \Asymp  -r_n\frac{((\sqrt{q} - 2)^2 -3)(\sqrt{q}-4)}{1-(4-\sqrt{q})^2} 
\]

Since the number $\frac{((\sqrt{q} - 2)^2 -3)(\sqrt{q}-4)}{1-(4-\sqrt{q})^2}$ is negative for $q \in (1,4)$, the centers $c(f_{q,n}^{-1})$ of the geodesics $\partial\hh_i(f_{q,n}^{-1})$ converge to $+\infty$ as $n$ goes to $+\infty$. 

We now study the radii : 

\[
\sqrt{I_{n,q}}\left( 1 + \frac{R_{n,q}^2}{(I_{n,q}-1)^2} \right)^{\frac{1}{2}} \Asymp \frac{1}{4 - \sqrt{q}}\frac{R_{q,n}}{I_{q,n} - 1} = -\frac{1}{4 - \sqrt{q}}c(f_{q,n}^{-1})
\]

where the number $\frac{1}{4 - \sqrt{q}}$ belongs to $(\frac{1}{3},\frac{1}{2})$. 

Thus, the extremities $e_l(f_{q,n}^{-1})$, for $l=1,2$, converge to $+\infty$. Moreover, since $\frac{1}{4 - \sqrt{q}} < \frac{1}{2}$, the geodesics $\partial\hh_i(f_{q,n}^{-1})$ are below the hypercycle $C_{(0,\infty)}(1+i)$ as claimed.\qed

\bibliographystyle{alpha}
\bibliography{horocycles_infinite_surfaces}

\vspace{1cm}

Institut de Recherche MAthématique de Rennes

Université de Rennes 1

263, Avenue du Général Leclerc

35042, Rennes

France

\textit{E-mail adress :} \href{mailto:alexandre.bellis@univ-rennes1.fr}{alexandre.bellis@univ-rennes1.fr}

\textit{Website :} \url{https://perso.univ-rennes1.fr/alexandre.bellis/}

\end{document}